\documentclass{cedram-aif}
\usepackage{amssymb,amsmath,amsthm,epsfig,bbm}
\usepackage{enumerate,multirow,psfrag,pb-diagram}
\usepackage{array}
\usepackage{subfig}
\usepackage{graphicx}
\usepackage{color}
\usepackage{alltt}
\usepackage{eqlist}
\usepackage[refpage]{nomencl}
\usepackage{notations}

\newcommand{\executeiffilenewer}[3]{%
\ifnum\pdfstrcmp{\pdffilemoddate{#1}}%
{\pdffilemoddate{#2}}>0%
{\immediate\write18{#3}}\fi%
}

\author{\firstname{Lionel} \middlename{F.} \lastname{Alberti}}
\address{McGill University,\\
Montreal, Quebec (Canada)}
\email{research@lionel.alberti.name}


\title{Polynomial Bound on the Local Betti Numbers of a Real Analytic Germ}

\begin{document}

\begin{abstract}

This article proves the existence of a bound on the sum of local Betti numbers of a real analytic germ by a polynomial function of its multiplicity. This result can be interpreted as a localization of the classical Oleinik-Petrovsky bound (also known as Thom-Milnor bound) on the sum of Betti numbers of a semi-algebraic set (see \cite{MR0161339}, \cite{Coste:SAGeom}
, thm 4.7, and also \cite{MR0200942, MR949442,MR1090179,MR1070358}). 

The proof relies on an interplay between geometric and algebraic arguments whose key elements are the tangent cone of the germ, the Thom-Mather topological trivialization theorem (\cite{stratif:Mather}, \cite{MR932724}), the Oleinik-Petrovsky bound, and a result by D. Mumford and J. Heintz (\cite{MR716823},  \cite{MR0282975}) bounding the degrees of the generators of an ideal by a polynomial function of the geometric degree of its associated variety.

Our result is then applied to yield bounds on invariants from singularity theory, such as the Lipschitz-Killing curvature invariants and the Vitushkin variations (which include the local density of a germ).

\end{abstract}

\maketitle
Experience from the complex case shows that multiplicity is key to controlling the local Betti numbers. In the real case, there are also results localizing the Oleinik-Petrovsky bound. In \cite{MR923549}, F. Loeser presents bounds in terms of the monodromy and homological invariants. In \cite{MR0247134}, R. N. Draper carries out a very thorough investigation of multiplicity from another angle, and characterizes it in several different manners, such as the Lelong number (see 
 theorem 7.3), and with a similar approach as ours, as the geometric degree of the tangent cone to the germ by looking at the local ring of the germ  (see 
 theorem 6.5). It is thus natural to think that the multiplicity could be used to bound the sum of local Betti numbers in the real case too. However, we show it is in fact not always possible to do so.

We first give a set of sufficient conditions on the dimension of the singular locus of the tangent cone for a polynomial bound in the multiplicity to exist (theorem  \ref{thm:main}). The proof is carried out in section \ref{sect:proof_main_thm} by applying the Thom-Mather topological trivialization theorem  (\cite{stratif:Mather}, \cite{MR932724}) to relate the topology of the germ to that of its tangent cone, followed by the Mumford-Heintz  theorem (\cite{MR716823},  \cite{MR0282975}) and the Oleinik-Petrovsky bound to control the sum of local Betti numbers in the tangent cone and thus the original germ. 

We then show that these conditions are actually optimal (theorem  \ref{thm:opti_main}) in section \ref{sect:coun_exam}  by exhibiting families of germs with constant multiplicity but arbitrary large $0^{th}$ Betti number whenever the conditions of theorem \ref{thm:main} are not met.

Finally in section \ref{sect:appl}, we apply our bound to interesting geometric invariants from singularity theory such as the local density of the germ (see \cite{MR1679984}
,
\cite{equ_sing_reelle_2}
), and with greater generality, the Lipschitz-Killing curvature invariants (\cite{MR1936341,equ_sing_reelle_2,MR2041428}, and definition \ref{defin:Lip-Kil_var} here) and the Vitushkin variations (\cite{MR0476953,MR2041428}, and definition \ref{def:Vitushkin_var} here). These quantities play an important role in defining notions of equisingularity for real varieties (\cite{MR1832990,equ_sing_reelle_2}). The local number of connected components in a generic affine section is also relevant to localizing results on entropy by means of the Vitushkin variations (see \cite{MR2041428}
, theorem 3.5). 
These quantities are linked to the local Betti numbers by an integral known as the local multidimensional Cauchy-Crofton formula (\cite{equ_sing_reelle_2} 
for the multidimensional case, and \cite{MR1679984}, 
 \cite{MR1832990} 
 for the original result about the density only). It inherits its name from the classical Cauchy-Crofton formula (see \cite{MR0022594}
 or \cite{MR0257325}
) which is a global quantity. The sum of Betti numbers bounds both the number of connected components which serves to define the Vitushkin variations and the Euler characteristic which serves to define the Lipschitz-Killing curvature invariants (the density being both a Vitushkin variation and a Lipschitz-Killing invariant). By the local multidimensional Cauchy-Crofton formula it is thus possible to derive sharp local bounds in terms of the multiplicity for the density, the Lipschitz-Killing invariants, and the Vitushkin variations.

\section{Statement of Main Results}\label{sect:main_thm}

\subsection{Local Betti Numbers}
We set the following convention:
\begin{itemize}
\item We will use the letter $X$ for real varieties,  $Z$ for complex varieties, and $Y$ for varieties that can be either real or complex.
\item The symbol $\Grass_{k,n}$ will denote the Grassmanian, and $\AffGrass_{k,n}$ the affine Grassmanian, which we will parametrize by $D\in\Grass_{k,n}$ and $t\in\orthog{D}/\{0\}$.
\item A property will be generically true iff it is true outside a set of measure zero (algebraically generic properties are thus \emph{a fortiori} generic).
\end{itemize}
\begin{defi}
Let $S\subset\R^n$ be a triangulable set. 
The $i^{th}$ Betti number $\Betti_i(S)$ of $S$ is the rank of the $i^{th}$ homology group of $S$.
\end{defi}
\begin{prop}[Local Betti Numbers]\label{def:numb_conn_comp}\index{Local Betti number}\nomenclature{$\Betti_i$}{Local Betti number}
Let $X\subset\R^n$ be an analytic germ at $0$, let $D\in\Grass_{k,n}$ and $t\in\orthog{D}/\{0\}$.\\
Define the local directional Betti number $\Betti_i(X,D,t)$  as
$$\Betti_i(X,D,t)=\lim_{\epsilon\to 0} \lim_{\lambda\to 0} \Betti_i(X\cap (\lambda t+D)\cap B_{0,\epsilon}),$$
where $\epsilon>0$ and $\lambda\in \R_+^*$.\\
Define the local Betti number $\Betti_i(X,k)$ as the infimum of the constants $c$ such that $c\ge \Betti_i(X,D,t)$ generically for $(D+\lambda t)\in \AffGrass_{k,n}$.

Then the limits defining $\Betti_i(X,D,t)$ always exist, and thus $\Betti_i(X,D,T)$ and $\Betti_i(X,k)$ are well-defined.
\end{prop}

\begin{proof}
The family $(X\cap (\lambda t+D)\cap B_{0,\epsilon})_{\lambda,\epsilon}$ is subanalytic. We can thus apply Hardt's theorem which shows that there are only finitely many topological types over $\lambda$ for each fixed $\epsilon$. The limit $\lim_{|\lambda|\to 0} \Betti_i(X\cap (\lambda t+D)\cap B_{0,\epsilon})$ thus exists for every $\epsilon$, and the function $\epsilon\mapsto \lim_{|\lambda|\to 0} \Betti_i(X\cap (\lambda t+D)\cap B_{0,\epsilon})$ is well-defined. Since limits can be written as first order logical formulas the graph of this function is also subanalytic, and thus analytic in a neighborhood of $0$. Since the function is discrete (values in $\N$) it is actually constant in this neighborhood, and the limit $\lim_{\epsilon\to 0} \lim_{|\lambda|\to 0} \Betti_i(X\cap (\lambda t+D)\cap B_{0,\epsilon})$ exists.
\end{proof}

\begin{defi}[Local Sum of  Betti Numbers]\nomenclature{$\SBetti$}{Local Betti numbers}
Let the local sum of directional Betti numbers for $D$ and $t$ be $$\SBetti(X,D,t)=\sum_{i=0}^n \Betti_i(X,D,t).$$
Let the local sum of Betti numbers $\SBetti(X,k)$ be the infimum of the constants $c\in\R$ such that $c\ge\SBetti(X,D,t)$ generically for $(D+\lambda t) \in \AffGrass_{k,n}$.
\end{defi}

\begin{rema}
Let $\SBetti'(X,k)=\sum_{i=0}^n \Betti_i(X,k)$, the sum of local Betti numbers. Note that $\SBetti(X,k)\ge\SBetti'(X,k)$ since $\SBetti'(X,k)$ is a sum of infima and is thus less than $\SBetti(X,k)$, the infimum of the sum. This is why we choose to control the local sum of Betti numbers rather than the sum of local Betti numbers.
\end{rema}

\subsection{Algebraic Tangent Cone}
As we mentioned earlier, for a real analytic germ $X$, it is not always possible to bound $\Betti_0(X,k)$ (hence $\SBetti$) in terms of the multiplicity. We introduce the notion of algebraic tangent cone that will serve to formulate the condition under which a bound exists, and define the multiplicity of a germ in relation to it.

In the following, $\K$ is a notation for $\C$ or $\R$. 
\begin{defi}
For any analytic germ at the origin $f\in\K((X_1,\ldots,X_n))$, let $\mu(f)$ denote the lowest total degree of the monomials of $f$, and let the initial part of $f$, ${\mathrm Init}(f)$, be the sum of the monomials of $f$ with degree $\mu(f)$.
\end{defi}

Since $\mu(f)$ is the lowest degree of the monomials of $f$, $\epsilon^{\mu(f)}$ can be factored out from the power series expansion of $f$ at $0$, and thus $f(\epsilon X_1,\ldots, \epsilon X_n)/\epsilon^{\mu(f)}$ is also an analytic germ in $\K((X_1,\ldots,X_n, \epsilon))$ (hence well-defined at $\epsilon=0$). Based on this observation, we can define the conic blow-up of $f$:

\begin{defi}[Conic Blow-up]\label{def:conic_blowup}\index{Conic blow-up}\nomenclature{$\Blow{X}_\epsilon$}{Conic blow-up}\nomenclature{$\Blow{Z}_\epsilon$}{Conic blow-up}
Let $Y\in\K^n$ be an analytic germ at $0$. Let $\J \subset \K((X_1,\ldots,X_n))$ be the ideal of functions vanishing on $Y$ .
We define the conic blow-up ideal as
$$\Blow{J}=\{f(\epsilon X_1,\ldots, \epsilon X^n)/\epsilon^{\mu(f)}\ |\ f\in\J\}\subset \K((X_1,\ldots,X_n,\epsilon)).$$
We call $V(\Blow{J})$ the conic blow-up of $Y$, and we conceive of it as a family of varieties in $\K^n$ over the dimension $\epsilon\in\K$, therefore we write
$$(\Blow{Y}_\epsilon)_{\epsilon\in\K} = V(\Blow{J}) \subset \K^n\times\K,$$
and $\Blow{Y}_\epsilon$ denotes the fiber in $\K^n$ defined by $\Blow{J}$ with $\epsilon$ fixed.
\end{defi}
One can note that $\Blow{Y}_0=V(\{{\mathrm Init}(f)\ |\ f\in\J\})$, and $\Blow{Y}_1=Y$. We will also use the fact that for all non-zero $\epsilon$ the $\Blow{Y}_\epsilon$'s are homothetic to each other.

\begin{defi}[Algebraic tangent cone]\label{def:alg_tang_cone} \index{Tangent cone}\index{Algebraic tangent cone|\\see{\mbox{Tangent cone}}}\nomenclature{$T(Y)$}{Algebraic tangent cone}
Let $Y\in\K^n$ be an analytic germ at $0$. The algebraic tangent cone $T(Y)$ of $Y$ is defined as the scheme $T(Y)=\Blow{Y}_0$.
\end{defi}

\begin{rema}
It is worth noting that although $Y$ is analytic, $T(Y)$ is in fact algebraic because the initial parts of the defining analytic functions are polynomials. The geometric significance of this observation is given by a result by H. Whitney in \cite{MR0188486} (theorem  5.8) which asserts that for a complex germ $Y$, the tangent cone $T(Y)$ is equal to the set of limits of secants of the germ at 0. However, for a real analytic germ, the algebraic tangent cone may contain additional points that do not come from limits of secants.
\end{rema}


We will repeatedly interpret cones as projective varieties and we introduce this definition to remove ambiguity:
\begin{defi}[Cone as Projective Variety]
Let $C\subset\K^n$ be a cone centered at $0$, then $C^\Proj$ denotes its associated interpretation as a subset of $\Proj^{n-1}(\K)$.
\end{defi}
We can now give a definition of the multiplicity based on the tangent cone. This approach stems from the work of R. N. Draper in \cite{MR0247134}.
\begin{prop}\label{prop:multproj}
Let $Z\subset\C^n$ be a complex analytic germ of dimension $d$ at $0$. 
Let $T$ be the algebraic tangent cone to $Z$, and let $T_d$ be the union of the components of $T$ of dimension $d$.\\
Then for all $D\in\Grass_{n-d+1,n}$ such that  $ T^\Proj_d\cap D\subset \Proj^{n-1}(\C)$ is a zero-dimensional variety, the dimension over $\C$ of the function sheaf of $\left(T^\Proj_d\cap D\right)$ is always equal to the same integer $\mu(Z)$.

We call $\mu(Z)$ the multiplicity of the germ $Z$. By extension, for $X\subset\R^n$ a real analytic germ, $\mu(X)$ denotes the multiplicity of the complexification of $X$.
\end{prop}
\begin{proof}
This is a straightforward consequence of the results in \cite{MR0247134}: theorem 6.4 shows that the degree of the projective variety $T^\Proj_d$ is equal to the degree of $Z$ at $0$, and then theorem 6.5 shows that the degree of $Z$ at $0$ is the multiplicity of the germ as defined classically by means of the Hilbert polynomial.
\end{proof}

At this point we emphasize the use of the word scheme in the definition of the tangent cone: the tangent cone is not necessarily defined by a radical ideal. The properties of this potentially non-radical ideal are key to proving our results. For instance, for $Z=V(x^2-y^3)$ in $\C^2$, $T(Z)=x^2$, and the multiplicity according to the previous definition is $2$, not $1$.

With this distinction in mind, let us  define the singular locus of the algebraic tangent cone:
\begin{defi}[Singularity of tangent cone]\index{Tangent cone!Singularity}\nomenclature{$\sing{T}$}{Singularity of tangent cone}
Let $T$ be the tangent cone for an analytic variety. We define $\sing{T}$ as the singular locus of $T$: the points $p$ where the local function sheaf at $p$ is not regular.
\end{defi}
Note that the previous definition is applicable whether $T$ is considered a cone in $\K^n$ or a variety in $\Proj(\K)^{n-1}$ and yields the same result. Also, because regularity of the local ring can be expressed as a condition on the rank of the Jacobian matrix, $\sing{T}$ itself is an algebraic variety.

\begin{rema}\label{rem:Grobner_multiplicity}
Effective calculation of the tangent cone for algebraic germs can be carried out using Gr\"obner bases (also known as standard bases). If $g_1,\ldots,g_l$ is a set of generators of $\I$ such that $(\mathrm{Init}(g_1),\ldots,\mathrm{Init}(g_l)) = \mathrm{Init}(\I)$, it is called a Gr\"obner basis of $\I$. Therefore the initial parts of the generators of the Gr\"obner basis are generators of the algebraic tangent cone. One of the most efficient algorithms to compute Gr\"obner bases is Faug\`ere's $F_5$ algorithm \cite{MR2035234}.
\end{rema}
\subsection{Main Theorem and Optimality Thereof}


We now  state our main result about the local sum of Betti numbers in a generic affine section of a real analytic germ (theorem  \ref{thm:main}) and the associated optimality theorem \ref{thm:opti_main}.
\begin{theorem}[Main Theorem]\label{thm:main}
Let $X\subset\R^n$ be a real analytic germ at $0$ of dimension $d$. Let $T$ be the algebraic tangent cone to $X$, and let $s=\dim\left(\sing{T}\right)$.
Then for any $k\in\N$ such that $2\le k \le n-1$
\begin{enumerate}
\item if $k<n-d$, for $A$ generic in $\AffGrass_{k,n}$, $A$ avoids $X$, and thus $\SBetti(X,k)=0$.
\item if $k=n-d$, we have the inequality $\Betti_0(X,k)\le\mu(X)$.\\
 Since $k=n-d$, $\forall i>0,\ \Betti_i(X,k)=0$ and the previous inequality can be rewritten $\SBetti(X,k)\le\mu(X)$.
\item if $n-d < k < n-s$ and the complexification of $X$ is pure dimensional, then 
$$\SBetti(X,k)\le \mu(X)(2\mu(X)-1)^{k-1}.$$
\end{enumerate}
\end{theorem}
\begin{rema}
Point (1) in theorem \ref{thm:main} can be seen as a generalization of lemma 1.4 in \cite{MR1607651}. This lemma shows that  $\Betti_0(X,k)\le\mu(X)$ for hypersurfaces and the main theorem generalizes it to the case of an arbitrary germ $X$.
\end{rema}
The following proposition shows that pure dimensionality of the complexification of a real germ is strictly weaker than pure dimensionality of the real germ itself:
\begin{prop}
Let $X\subset\R^n$ be a real analytic germ at $0$.
\begin{enumerate}
\item If $X$ is pure dimensional of dimension $d$, then its complexification $\Complex{X}$ is also pure dimensional of dimension $d$.
\item The real germ defined by $f(x,y,z)=x^4+4x^3 z+4x^2 z^2-y^2 z^2$ is not pure dimensional but its complexification is.
\end{enumerate}
\end{prop}
\begin{proof}[Proof of (1)]
For any point $p$ and an analytic variety $Y$, we define the vector space of differentials $D(Y,p)=\{ D_p(f)\ |\  f \in V(Y)\}$. Since the defining functions of $\Complex{X}$ and $X$ are the same, at any point $p\in X,\ D(X,p)\otimes \C=D(\Complex{X},p)$ and $\dim_\R \left(D(X,p)\right)=\dim_\C \left(D(\Complex{X},p)\right)$. In particular at any smooth point $p$, the codimension of a variety is the dimension of the vector space of differentials, so the dimensions of $\Complex{X}$ and $X$ are the same at any smooth point.

Let $Z$ be any irreducible component of $\Complex{X}$. Since the complexification is the minimal complex variety containing $X$, $Z$ must intersect $\R^n$. Since $Z \cap \R^n \subset X$ is itself an analytic variety it contains a smooth point $p$ in $X$. By what precedes the dimensions of $Z$ and $X$ are the same.
\end{proof}
\begin{proof}[Proof of (2)]
The graph of $f(x,y,1)$, has two humps: one emerges over the plane $f=0$ and the other one is tangent to it, which creates an oval and a point. However, the complex variety is pure dimensional as $f$ is irreducible (this can be shown elementarily by attempting a factorization in $y$, since $\deg_y f$ is only 2).
\end{proof}

The following theorem precisely states how our main theorem \ref{thm:main} is optimal:
\begin{defi}[Counter-example Family]\label{def:coun_fam}
For any $(n,k,s)\in\N^3$ we say that a family of germs $(X_l)_{l\in\N}$ is a counter-example family of type $(n,k,s)$ if and only if $\forall l,\ X_l \subset \R^n$, $\dim\left(\sing{T(X_l)}\right)=s$, $\mu(X_l)_{l\in\N}$ is bounded by a constant, and \mbox{$\lim_{l\to\infty}\Betti_0(X_l,k)=+\infty$}\\
As a short-hand we will say that a counter-example family is pure dimensional if the \emph{complexifications} of the germs in the family are all pure dimensional.
\end{defi}
Note that the definition imposes a condition on $\Betti_0$ rather than a weaker one on $\SBetti$ only because our counter-examples only require $\Betti_0$.
\begin{theorem}[Optimality of Main Theorem]\label{thm:opti_main}
For any $n,k\in\N$ such that $2\le k \le n-1$,
\begin{enumerate}
\item for any $s\in\N$ such that $s\le n-1$, there exists a counter-example family of type $(n,k,s)$ that is not pure dimensional.
\item for any $s\in\N$ such that $n-k\le s \le n-1$, there exists a pure dimensional counter-example family of type $(n,k,s)$.
\end{enumerate}
\end{theorem}
\begin{rema}
The optimality theorem \ref{thm:opti_main} complements theorem \ref{thm:main} exactly since for any \mbox{$s\in\{0,\ldots,n-1\}$}, theorem \ref{thm:main} covers pure dimensional counter-example families for all $k<n-s$, and theorem \ref{thm:opti_main} covers all families that are either not pure dimensional, or pure dimensional with $k\ge n-s$.
\end{rema}

\section{Proof of Main Theorem}\label{sect:proof_main_thm}
We now prove the main theorem \ref{thm:main}. Point (1) of the main theorem is self-evident since generic sections by affine space of dimension less than the codimension of the germ are generically empty. In subsection \ref{sect:0dim_intersec}, we begin by proving point (2), where the dimension of the intersecting affine spaces is equal to the codimension of the germ, by using our careful characterization of multiplicity based on the tangent cone. We then prove point (3), where the dimension of the intersecting affine spaces is greater than the codimension of the germ, by showing that for a generic affine space, the topology of the intersection with the tangent cone is the same as the topology of the intersection with the germ (subsection \ref{sect:aff_section_tan_con_germ_topo_idem}), hence their Betti numbers are the same. The equivalence between both topologies is proved using the Thom-Mather topological trivialization theorem (see \cite{stratif:Mather} or \cite{MR932724}). Then, we use proposition \ref{lem:multtancone} to prove that the multiplicity of the germ is equal to the geometric degree of the sections of the tangent cone. This finally allows us to use the Heintz-Mumford result (theorem \ref{thm:mum_geom_alge_deg}) to bound the degree of the generators of the sections of the tangent cone by the multiplicity, and thereby to bound the sum of the Betti numbers by a polynomial function of the multiplicity using the Oleinik-Petrovsky bound  (theorem \ref{CCupperbound}).

\subsection{Proof of point (2) of the main theorem}\label{sect:0dim_intersec}
It is a basic fact that the multiplicity of a complex germ is equal to the number of points close enough to the origin that lie in the intersection of the germ with a generic affine space whose dimension is the codimension of the germ. Therefore, point (2) of the main theorem \ref{thm:main} could seem straightforward. However, we are considering real germs, and real affine spaces are contained in a set of complex affine spaces of measure zero, so generic equality over $\C$ may not entail generic equality over $\R$. We thus need a slightly more refined characterization of the multiplicity than this to conclude.\\ 
We introduce such a characterization as lemma \ref{lem:multtancone}, which is formulated as geometric conditions on the tangent cone so that we can apply it to prove point (2) of the main theorem as corollary \ref{cor:zero_dim_bound}, as well as to prove point (3) of the main theorem later on in the proof of theorem \ref{thm:main_smooth_case}. 

\begin{defi}
Let $Z\in\C^n$ a Zariski zero-dimensional scheme, then $\dim_\C{\mathcal O}(Z)$ denotes the vector space dimension over $\C$ of $\funsheaf{Z}$, the function sheaf of $Z$. In other words, $\dim_\C{\mathcal O}(Z)$ counts points in $Z$ with multiplicity.\\
 This is the same as the length of $\funsheaf{Z}$ localized at $p$ over the maximal ideal at $p$ , which is the classical definition of module multiplicity algebraically (see \cite{MR0463157}, chapter 7 for instance).
\end{defi}

\begin{lemma}\label{lem:multtancone}
Let $Z\in\C^n$ be a complex germ at $0$ of dimension $d$ and multiplicity $\mu$. 
Let $T=\Blow{Z}_0$ be the tangent cone to $Z$, and let $T_{<d}$ be the union of the components of the tangent cone with dimension less than $d$.
Let $k=n-\dim\left(Z\right)$. Let $D\in\Grass_{k,n}$ and $t\in\orthog{D}/\{0\}$.
Assume $D\cap T=\{0\}$,  $(t+D)\cap T$ is \mbox{$0$-dimensional} and $(t+D)\cap T_{<d}=\emptyset$ in $\C^n$.
Then 
\begin{displaymath}\begin{split}
\exists& \epsilon_0>0,\ \text{s.t. }\forall \epsilon\in]0,\epsilon_0[,\ \exists\lambda_0>0,\ \text{s.t. }\forall \lambda\in]0,\lambda_0[,\\ 
&\dim_\C{\mathcal O}\big( (\lambda t+D)\cap B_{0,\epsilon}\cap Z\big)\ = \dim_\C{\mathcal O}\big( (t+D)\cap T\big) = \mu
\end{split}\end{displaymath}
\end{lemma}
\begin{proof}
Consider  $Z'_\epsilon=(\C t+D)\cap B_{0,1}\cap \Blow{Z}_\epsilon$. For $\epsilon\neq 0$, it is homothetic to $(\C t+D) \cap B_{0,\epsilon} \cap Z$ because $\epsilon\Blow{Z}_\epsilon=\Blow{Z}_1=Z$, $\epsilon B_{0,1}=B_{0,\epsilon}$ and $(\C t+D)$ is a vector space hence is invariant by homothety. 
Since  $(t+D)\cap T$ is \mbox{$0$-dimensional}, for $\epsilon<\epsilon_0$ small enough $Z'_\epsilon$ consists of points and $1$-dimensional branches only. By possibly reducing $\epsilon_0$ further, we can ensure that $Z'_\epsilon$ only contains $1$-dimensional branches that are all connected to the origin. Therefore, each of these branches will give rise to a point in the projective variety $(\C t+D) \cap T$. Since by hypothesis $D\cap T=\{0\}$, there are no points at infinity, and each branch is thus transverse to the $(\lambda t+D)$ for $\lambda<\lambda_1$ small enough. Therefore, branches of $Z'_\epsilon$ are in one to one correspondence with the points in $ (\lambda t+D)\cap B_{0,1}\cap \Blow{Z}_\epsilon = (\lambda\epsilon t + D)\cap B_{0,\epsilon}\cap Z$, and each branch gives rise to one point in $(t+D)\cap T$ in $\C^n$. Since $(t+D)\cap T_{<d}=\emptyset$ by hypothesis, all the points in $(t+D)\cap T$ come from a branch of $Z'_\epsilon$, and there is also a one to one correspondence between branches of $\Blow{Z}_\epsilon$ and points in $(t+D)\cap T$. This proves the first equality for $\epsilon_0$ and $\lambda_0=\epsilon_0\lambda_1$.

The equality with the multiplicity stems from the fact that $(t+D)$ avoids $T_{<d}$ by hypothesis, so $ \dim_\C{\mathcal O}\big( (t+D)\cap T\big)=\dim_\C{\mathcal O}\big( (t+ D)\cap T_d \big)$, and since $D\cap T=\{0\}$, $\dim_\C \funsheaf{(t+D)\cap T_d}=\dim_\C \funsheaf{(\C t+D)\cap T^\Proj_d}$, which is equal to $\mu$ by proposition \ref{prop:multproj}.
\end{proof}

The previous lemma allows us to see that the multiplicity of the germ can be obtained using a limit of intersecting affine spaces chosen in a Zariski open set of the Grassmannian, and not only chosen in an unspecified generic subset of it. This added structure allows us to prove the first item of the main theorem:
\begin{coro}\label{cor:zero_dim_bound}
Let $X \subset \R^n$ be a real analytic germ at $0$ of dimension $d$. Then $\SBetti(X,n-d)=\Betti_0(X,n-d)\le\mu$.
\end{coro}
\begin{proof}
As the intersection of $X$ and a generic affine space of dimension $(n-d)$ is generically zero-dimensional, $\SBetti(X,n-d)=\Betti_0(X,n-d)$.

The conditions on $D$ and $t$ in lemma \ref{lem:multtancone} are all satisfied on a \emph{Zariski} open set of the affine complex Grassmannian $\AffGrass_{n-d,n}(\C)$. Since the real affine Grassmannian is a subscheme of the complex one, the lemma applies to a Zariski open set of the real affine Grassmannian, which is always generic in the \emph{real} Grassmannian $\AffGrass_{n-d,n}(\R)$. Therefore by lemma \ref{lem:multtancone}, the complexification of a generic real affine section of the germ contains $\mu$ points, and thus this generic real affine section cannot contain more than $\mu$ points. 
\end{proof}

\subsection{Sections of tangent cone and germ have the same topology}\label{sect:aff_section_tan_con_germ_topo_idem}
We know that analytic sets admit Whitney stratifications (\cite{MR0192520}, 
theorem 19.2). In the following it is understood that we stratify the sets we consider with an arbitrary Whitney stratification the first time we encounter them and all the subsequent proofs are carried out with respect to these arbitrary stratifications.

This section shows that we can relate the topology of the tangent cone to that of the germ, and therefore equate their Betti numbers.
\begin{theorem}\label{thm:tang_germ_topo_idem}
Let $X\in\R^n$ be a real analytic germ at $0$ and $\Blow{X}_\epsilon$ its conic \mbox{blow-up} at the origin (def. \ref{def:conic_blowup}).
Let $T$ be the algebraic tangent cone to $X$. For any element $D\in\Grass_{k,n}$ and $t\in\orthog{D}$, assume that
\begin{enumerate}
\item $T$ and $D$, as stratified projective varieties, are transverse (i.e. the sum of their tangent spaces at intersection points is of maximal dimension, that is $n-1$).
\item $T$ and $(\R t+ D)$, as stratified projective varieties, are transverse.
\item The projective variety $(\R t+ D)$ does not intersect $\sing{T^\Proj}$.
\end{enumerate}
Then we have
\begin{displaymath}\begin{split}
\exists\lambda_0>0,&\ \exists \epsilon_0>0,\ \forall\epsilon\in]0,\epsilon_0],\ \forall\lambda\in]0,\lambda_0],\ \exists h\text{ homeomorphism},\\ 
	 &(t+D)\cap \Blow{X}_0 \ \stackrel{h}{\approx}\ (\lambda t+ D)\cap B_{0,1}\cap\Blow{X}_\epsilon
\end{split}\end{displaymath}
\end{theorem}

The proof hinges on the Thom-Mather topological trivialization theorem. We use a variant known as the moving the wall theorem (\cite{MR932724}, theorem 1), which we recall here
\begin{theorem}[Moving the Wall
]\label{thm:movi_wall_orig}
Let $X\subset\R^n$ be a Whitney stratified set, and
let $f:\R^n \to \R$ be a smooth map.
Let $X_t := f^{-1}(t)\cap X$, and for any stratum $\sigma$ of $X$ and $t\in\R$, let $\sigma_t:=f^{-1}(t)\cap \sigma$.

We endow $X_0\times\R$ with the stratification 
${\mathcal S}_0=\left\{\sigma_0\times\R\ |\ \sigma\text{ stratum of } X\right\}.$
If $f|_X$ is proper, and $\forall t\in\R$, $f$ is a submersion on each stratum of $X$,
then ${\mathcal S}_0$ is a Whitney stratification for $X_0\times\R$ and there exists a homeomorphism
$h: X\ \to\  X_0\times\R,$
which is smooth on the strata of $X$ and such that $\Pi_2\circ h= f$ with $\Pi_2$ the projection to $\R$, the second component. In addition $h$ is stratum preserving, that is, for any stratum $\sigma$, we have $h(\sigma)\subset \left(\sigma_0\times\R\right)$ (note that $(\sigma_0\times\R)\ \in {\mathcal S}_0$).
\end{theorem}

We introduce the following definitions to carry out the proof of theorem \ref{thm:tang_germ_topo_idem}
\begin{defi}\nomenclature{$C_\lambda$}{$\lambda$-cylinder}\nomenclature{$D_\lambda$}{$\lambda$-disk}\nomenclature{$H_\lambda$}{$\lambda$-half sphere}
For any $D\in\Grass_{k,n}(\R)$, $t\in\orthog{D}$ and $\lambda,\kappa\in\R$, we define:
\begin{align*}
C_\lambda&= (\lambda t+ D)\cap S_{0,1},\\
D_\lambda&= (\lambda t+ D)\cap B_{0,1},\\
H_\lambda&= ( (\lambda + \R_+^*) t+ D)\cap S_{0,1} = \cup_{\kappa>\lambda} C_\kappa ,\\
F_{\lambda,\kappa}&= \left\{p\in(\lambda+\R_+^*)t+D \mid (1-\kappa)\left(\Enorm{p}^2-1\right) + \kappa\left((p\lvert t)-\lambda\Enorm{t}^2\right)=0\right\}\\
&\text{ where $(.|.)$ is the scalar product, and $\Enorm{.}$ is the Euclidian norm.}
\end{align*}
We also define the topological closures of these sets: $\overline{D_\lambda}=D_\lambda \cup C_\lambda$, $\overline{H_\lambda}=H_\lambda \cup C_\lambda$, and $\overline{F_{\lambda,\kappa}} = F_\lambda \cup C_\lambda$.
\end{defi}
The formal definition of $F_{\lambda,\kappa}$ may appear a bit obscure, but $F_{\lambda,\kappa}$ is simply a family of spherical sectors deforming $F_{\lambda,0}=H_\lambda$ into $F_{\lambda,1}=D_\lambda$  while leaving $C_\lambda$ fixed. The $F_{\lambda,0},F_{\lambda,0.5},$ and $F_{\lambda,1}$ are depicted in figure \ref{fig:spherical_to_planar_slice} as part of the homeomorphism $h_3$.

\begin{proof}[\textit{Proof of theorem \ref{thm:tang_germ_topo_idem}}]
The complete homeomorphism is built in four steps:
$$(t+D)\cap \Blow{X}_0 \ \stackrel{h_1}{\approx}\ H_0(t,D)\cap \Blow{X}_0 \ \stackrel{h_2}{\approx}\ H_\lambda(t,D)\cap\Blow{X}_0 \ \stackrel{h_3}{\approx}\ D_\lambda(t,D)\cap\Blow{X}_0\ \stackrel{h_4}{\approx} D_\lambda(t,D)\cap\Blow{X}_\epsilon$$

\begin{figure}[ht]
\centering  \def\svgwidth{\columnwidth}  %
\executeiffilenewer{spherical_to_planar_slice.svg}{spherical_to_planar_slice.pdf}%
{inkscape -z -D --file=spherical_to_planar_slice.svg %
--export-pdf=spherical_to_planar_slice.pdf --export-latex}%
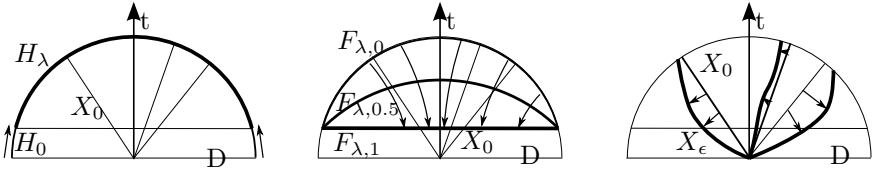%
 
\caption{From left to right, $h_2$ retracts the hemisphere $H_0$ up to $H_\lambda$. Then $h_3$ flattens $H_\lambda(\!=\! F_{\lambda,0})$ to $D_\lambda(\!=\! F_{\lambda,1})$ (an intermediate transition $F_{\lambda,0.5}$ is also depicted). Finally $h_4$ perturbs the tangent cone $\Blow{X}_0$ into $\Blow{X}_\epsilon$, which is homothetic to the original germ $X$.}
\label{fig:spherical_to_planar_slice}
\end{figure}

We define $h_1$ as $$h_1:\ p\in(t+D)\mapsto \frac{p}{\Enorm{p}}\in H_0(t,D).$$
This is clearly a homeomorphism, and since $\Blow{X}_0$ is a cone, $h_1$ leaves $\Blow{X}_0$ stable, and is thus a homeomorphism between $(t+D)\cap \Blow{X}_0$ and $H_0(t,D)\cap \Blow{X}_0$.

\begin{figure}[ht]
\centering  \def\svgwidth{\columnwidth}  %
\executeiffilenewer{spherical_to_planar_slice_obstructions.svg}{spherical_to_planar_slice_obstructions.pdf}%
{inkscape -z -D --file=spherical_to_planar_slice_obstructions.svg %
--export-pdf=spherical_to_planar_slice_obstructions.pdf --export-latex}%
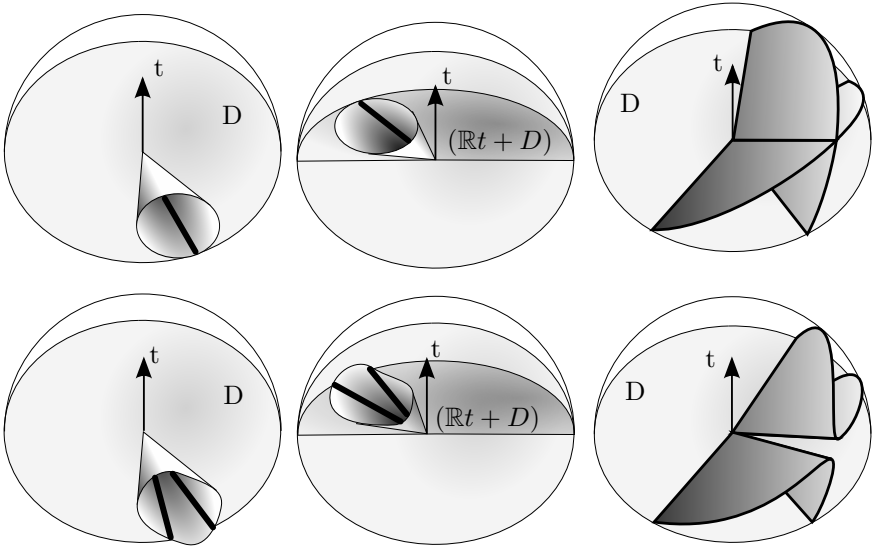%
 
\caption{The obstructions are shown on top with the unwanted change in topology that may result below. From left to right, they correspond to conditions 1, 2 and 3 in theorem \ref{thm:tang_germ_topo_idem}: the tangent cone is tangent to $D$, tangent to $\R t+D$, or singular. In the left and middle column, a line splits into two, and in the right column the surface splits into two wings.}
\label{fig:spherical_to_planar_slice_obstructions}
\end{figure}

We construct the homeomorphisms $h_2$, $h_3$ and $h_4$ depicted in figure \ref{fig:spherical_to_planar_slice} by applying the moving the wall theorem to:
$$f_2=\left\{(\overline{H_\lambda} \cap \Blow{X}_0, \lambda) \subset (\R t+D)\times \R \mid \lambda\in\R\right\} \mapsto \lambda$$
$$f_3=\left\{(\overline{F_{\lambda,\kappa}} \cap \Blow{X}_0, \kappa) \subset (\R t+D)\times \R  \mid \kappa\in\R\right\} \mapsto \kappa$$
$$f_4=\left\{(\overline{D_\lambda} \cap \Blow{X}_\epsilon, \epsilon) \subset \R^n \times \R  \mid \epsilon\in\R\right\} \mapsto \epsilon$$
According to the moving the wall theorem we need only stratify the source spaces of these functions and show they are proper submersions on each stratum to prove the existence of the required homeomorphisms. The possible obstructions that can prevent the functions $f$ from being submersive are depicted in figure \ref{fig:spherical_to_planar_slice_obstructions} and correspond to conditions 1,2, and 3 in theorem \ref{thm:tang_germ_topo_idem}. It is key to notice that all the functions $f$ are globally defined for parameters in $\R$, and not just for non-negative $\lambda$, $\kappa$ and $\epsilon$: This is how we prove that $f_4$ is a submersion at $0$.

By condition 3 of theorem \ref{thm:tang_germ_topo_idem}, $T=\Blow{X}_0$ is smooth in a neighborhood of $(\R t+ D)$. By condition 2, $\Blow{X}_0$ and $(\R t+D)$ are transverse and $\Blow{X}_0\cap(\R t+D)$ is thus also smooth. Therefore the constant family $(\Blow{X}_0\cap(\R t+D)) \times \R \subset (\R t + D)\times \R$ consists of a single smooth stratum.

On the other hand the families $\overline{H_\lambda}$, and $\overline{F_{\lambda,\kappa}}$ are Whitney stratified by their interiors ($H_\lambda$ and $F_{\lambda,\kappa}$) and $C_\lambda$.
To apply the moving the wall theorem to $f_2$ and $f_3$, it thus suffices to show that the stratifications of these latter families are transerve to $(\Blow{X}_0\cap(\R t+D)) \times \R\subset (\R t + D)\times\R$.

For $f_2$ the two strata to consider are $C_\lambda$ and $H_\lambda$. For $C_\lambda$, condition 1 in theorem \ref{thm:tang_germ_topo_idem} ensures that $C_0$ is transverse to $T=\Blow{X}_0$, and since transversality is an open condition and we work in a compact set $C_0$, there exists $\lambda_0>0$ such that $C_\lambda$ is transverse to $\Blow{X}_0$ for all $\lambda \in [0,\lambda_0]$. Since $\Blow{X}_0$ is a cone, the tangent space at $p\in\Blow{X}_0$ contains the direction $(p-0)$ which is never contained in the tangent space at $p\in H_\lambda$. Since $H_\lambda$ is of codimension $1$ in $(\R t+D)$, the sum of the tangent spaces to $\Blow{X}_0\cap (\R t+D)$ and to $H_\lambda$ is of maximal dimension, that is $\Blow{X}_0\cap (\R t+D)$ and $H_\lambda$ are transverse for all $\lambda\ge 0$. Therefore $C_\lambda \cap \Blow{X}_0$ and $H_\lambda \cap \Blow{X}_0$ stratify the source space of $f_2$, and since projection to the parameter of a family ($\lambda$ here) is automatically a submersion, $f_2$ is submersive on this stratification. The moving the wall theorem thus asserts the existence of $h_2$.

For $f_3$, we consider the strata $C_\lambda$ and $F_{\lambda,\kappa}$.
As we just mentioned $C_\lambda$ is transverse to $\Blow{X}_0$ for $\lambda$ small enough, and since $C_\lambda$ is fixed as the parameter $\kappa$ varies this transversality is retained and $C_\lambda\cap \Blow{X}_0$ is a stratum.
In a similar way as for $H_\lambda$ above, the sets $F_{\lambda,\kappa}\ (\kappa\in[0,1])$ are all transverse to $\Blow{X}_0$ in $(\R t + D)$ since $\Blow{X}_0$ is a cone.
Therefore $F_{\lambda,\kappa}\cap\Blow{X}_0$ and $C_\lambda\cap\Blow{X}_0$ stratify the source space of $f_3$, and $f_3$ is a submersion because it projects to the parameter $\kappa$ on these strata. This shows the existence of $h_3$ by the moving the wall theorem.

For $f_4$ we need to consider the strata $C_\lambda$ and $D_\lambda$ that make up $\overline{D_\lambda}$. 
As $\lambda$ was picked small enough earlier so that $C_\lambda$ and $D_\lambda$ are transverse to $\Blow{X}_0$, and since transversality is an open condition and $\overline{D_\lambda}$ is compact, $\Blow{X}_\epsilon$ will remain transverse to $C_\lambda$ and $D_\lambda$ in $(\R t+D)$ for all $\epsilon\in(-\epsilon_1,\epsilon_1)$ for $\epsilon_1$ small enough.
But by condition 2, $(\R t+D)$ and $\Blow{X}_0$ are transverse, so by reducing $\epsilon_1$ further to $\epsilon_0>0$, we can ensure that $\Blow{X}_\epsilon$ is transverse to $C_\lambda$ and $D_\lambda$ in $\R^n$ for $\epsilon \in (-\epsilon_0,\epsilon_0)$.

Therefore $C_\lambda \cap \Blow{X}_\epsilon$ and $D_\lambda \cap \Blow{X}_\epsilon$ ($\epsilon\in(-\epsilon_0,\epsilon_0)$) form a Whitney stratification of the source space of $f_4$, and $h_4$ exists by the moving the wall theorem.
\end{proof}
\begin{coro}
Let $X\in\R^n$ be a real analytic germ at $0$ and $T$ its tangent cone. Let $\Betti_i$ be any local Betti number.
For any elements $D\in\Grass_{k,n}$ and $t\in\orthog{D}$, if conditions (1),(2) and (3) in theorem \ref{thm:tang_germ_topo_idem} are satisfied, then 
$$\Betti_i(X,D,t)=\Betti_i(T,D,t).$$
\end{coro}
\begin{proof}
We have:
\begin{align*}
\Betti_i(X,D,t)&= \lim_{\epsilon\to 0}\lim_{\lambda\to 0} \Betti_i((\lambda t+ D)\cap B_{0,\epsilon}\cap X) &\text{by definition}\\
&= \lim_{\epsilon\to 0}\lim_{\lambda\to 0} \Betti_i((\lambda/\epsilon t + D)\cap B_{0,1}\cap \Blow{X}_\epsilon) &\text{by $1/\epsilon$-homothety}\\
&= \Betti_i((t+D)\cap\Blow{X}_0). &\text{ by theorem  \ref{thm:tang_germ_topo_idem}}
\end{align*}
As $T$ is already a cone, its conic \mbox{blow-up} is the constant family $\Blow{T}_\epsilon=\Blow{X}_0=T$, and thus  $\Betti_i((t+D)\cap\Blow{X}_0)=\Betti_i(T,D,t)$, which completes the proof.
\end{proof}
We use the previous corollary, which is dependent on a direction $D$, to yield the desired relation between the Betti numbers of the germ and of its tangent cone, which is direction-independent. 
\begin{coro}\label{cor:tang_germ_mult_equa}
For $X\subset\R^n$ be a real analytic germ with pure dimensional complexification. Let  $T$ be the tangent cone to $X$, and let $s$ be the dimension of $\sing{T}\subset\R^n$. Let $k\in\N$ such that $k+s<n$, and $\Betti$ be any function of the local Betti numbers $(\Betti_i(.,k))_{i\in\N}$. Then
$\Betti(X,k)=\Betti(T,k).$
\end{coro}
\begin{proof}
It suffices to show that the three conditions of theorem \ref{thm:tang_germ_topo_idem} are satisfied for $D\in\Grass_{k,n}$ generic and $t\in\orthog{D}$ generic, because then all the directional local Betti numbers will be equal at once.

We first consider conditions 1 and 2. It is true in general that for $L\in\Grass_{m,n}$  and for $W^\Proj$, a Whitney stratified projective analytic variety, $L^\Proj$ and $W^\Proj$ are generically transverse.
Applying it to $L\in\Grass_{k+1,n}$ and $W^\Proj=T^\Proj$ we have that $L^\Proj$ and $T^\Proj$ are transverse for $L$ in a set $G\subset\Grass_{k+1,n}$ of measure one. Therefore for $D\in\Grass_{k,n}$ in a subset $G_2\subset\Grass_{k,n}$ of measure one, we have $(\R t+D)\in G$  for $t$ generic in $\orthog{D}$. In other words, condition 2 of theorem \ref{thm:tang_germ_topo_idem} is met for $D\in G_2$ and $t$ generic in $\orthog{D}$.

Similarly, for $D\in\Grass_{k,n}$ and $W^\Proj=T^\Proj$ we have that $D$ and $T^\Proj$ are transverse over a set $G_1\subset\Grass_{k,n}$ of measure one.That is, condition 1 is met for $D\in G_1$. Since $G_1\cap G_2\subset \Grass_{k,n}$ is of measure one as the intersection of two measure one sets, conditions 1 and 2 are met for $D\in\Grass_{k,n}$ generic and $t\in\orthog{D}$ generic.

For condition 3, we have $\dim\left( (\R t+D)^\Proj\right)=k$ and $\dim(\sing{T^\Proj})=s-1$ as projective varieties. Since $k+s<n$, $\dim\left( (\R t+D)\right) +\dim\left(\sing T^\Proj\right)=k+s-1<n-1$, we know that $(\R t+D)$ generically avoids $\sing{T}$ in $\Proj^{n-1}$. That is, condition 3 is generically true.
\end{proof}

\subsection{Proof of point (3) of main theorem}\label{subsect:proof_main_thm}

 In order to finish the proof of the main theorem we need two more auxiliary results. The first one is the well-known Oleinik-Petrovsky bound on the number of connected components of a real (affine or projective) algebraic variety (see \cite{MR0161339},proof of theorem 2, or \cite{MR1090179,MR949442}).
\begin{theorem}[Oleinik-Petrovski/Thom-Milnor bound]\label{CCupperbound}
If $X$ is an algebraic variety defined by polynomials of degree at most $d$ in $\R^n$ or $\Proj^n(\R)$ then
$$\SBetti(X)\le d(2d-1)^{n-1}.$$
\end{theorem}
The second auxiliary theorem was proved by J. Heintz (\cite{MR716823}
prop.3). D. Mumford also obtained a similar result that appears in the proof of theorem 1, in \cite{MR0282975} (the proof is less detailed than in J. Heintz's presentation). It states that if $\mu$ is the geometric degree of a pure dimensional affine variety, then its defining ideal is generated in degree $\mu$.
\begin{defi}[Geometric Degree]\index{Geometric degree}
Let $Z\subset\C^n$ be a variety of dimension $d$ such that the intersection with a generic $(n-d)$ affine space contains $\delta(Z)$ points (counted with multiplicity). Then $\delta(Z)$ is the geometric degree of $Z$.
\end{defi}
We will use the following straightforward generalization of their result:
\begin{theorem}[Generators in bounded degree ]\label{thm:mum_geom_alge_deg}
Let $Z\subset\C^n$ be a pure dimensional variety with geometric degree $\delta$. Then there exist finitely many generators $(g_i)_{i\in I}\in\funsheaf{\C^n}$ such that $\deg g_i\le \delta$ ($\forall i\in I$), and $Z=V((g_i)_{i\in I})$.
\end{theorem}
\begin{proof}
J. Heintz's result in \cite{MR716823} (
prop.3) states that for an irreducible variety \mbox{$Y\subset\C^n$}, there exist an ideal $\I=(g_1,\ldots,g_{n+1}) \subset \funsheaf{\C^n}$ such that $Y=V(\I)$ and $\deg g_i\le \delta(Y)$ where $\delta(Y)$ is the geometric degree of $Y$.\\
Since $Z$ is equidimensional, it can be broken down into its irreducible components $V(\I_i)$ so that $Z=V(\prod_i \I_i)$. It is clear from the definition of the geometric degree that $
\delta( Z) = \sum_i \delta(V( \I_i))$. And by applying J. Heintz's result to each irreducible component we obtain generators $g_{ij}$ such that $\I_i=(g_{i1}, \ldots, g_{ik_i})$ with $\deg g_{ij}\le\delta(V(\I_i))$. Consequently, $V(Z)$ is generated by products $\prod_i g_{ij_i}$, and each of these products is of degree $\sum_i\delta(V(\I_i))=\delta(Z)$.
\end{proof}

Equipped with these last two theorems we can finally prove the third item of the main theorem \ref{thm:main}.

\begin{theorem}\label{thm:main_smooth_case}
For $X\subset\R^n$ an analytic germ of pure dimension $d$, let $\mu$ be the multiplicity $\mu(X)$ of $X$, and let $s$ be the dimension of the singular locus of $T(X)$. If $k+s < n$, then
$\SBetti(X,k)\le \mu(2\mu-1)^{k-1}.$
\end{theorem}
\begin{proof}
The conditions that $s+k<n$ and $X$ is pure dimensional are the conditions required to apply corollary \ref{cor:tang_germ_mult_equa}. By applying it to the sum of Betti numbers, we obtain $\SBetti(X,k)=\SBetti(T,k)$. We now bound $\SBetti(T,k)$.

First we show that generic sections of $T$ have geometric degree $\mu$. Let $D\in\Grass_{k,n}(\R)$ and $t\in\orthog{D}/\{0\}$. 
Let $T_\C$, $t_\C$ and $D_\C$ be the complexifications of $T$, $t$ and $D$, and let $S = (t_\C+D_\C)\cap T_\C$.
Let $T_{<d}$ be the union of components of $T_\C$ of dimension less than $d$. As $X$ is pure $d$-dimensional, $T_{<d}\subset \sing{T}$.  The condition $k+s<n$ thus implies that $\dim T_{<d}<n-k$, and thus for generic $D_\C$ and $t_\C$, $(t_\C+D_\C)$ avoids $T_{<d}$. Furthermore, since $(t_\C+D_\C)$ is generically transverse to $T_\C$, we have that $S$ is a pure dimensional complex variety of dimension $k+d-n$ in $t_\C+D_\C \approx \C^k$.

Now that we have picked a suitable generic section $S$, we show its geometric degree is $\mu$. Let $D'\in\Grass_{n-d,n}(\C)$, and $t'\in \orthog{D'}$ generic such that $t'+D'\subset t_\C+D_\C$. As $D$ is generic and $D'\subset D$ we can assume that $T\cap D'=\{0\}$, and $(t'+D')\cap T$ is Zariski \mbox{$0$-dimensional}. Also $(t'+D')\cap T_{<d}=\emptyset$ as $(t_\C+D_\C)$ avoids $T_{<d}$. Therefore we can apply lemma \ref{lem:multtancone} to $T, t',$ and $D'$, and we have $\mu=\dim_\C{\mathcal O}((t'+D')\cap T)$. Since $(t'+D')\cap T = (t'+D')\cap S$, we generically have $\mu=\dim_\C{\mathcal O}((t'+D')\cap S)$, that is the geometric degree of $S$ is $\mu$.

As we have made sure that $S$ is pure dimensional, we can apply the Heintz-Mumford result (theorem  \ref{thm:mum_geom_alge_deg}), and conclude that there exist generators of the defining ideal of $S$ in degree $\mu$. Since $S$ is a complexification, $S\cap \R^n = (t+D)\cap T$, and the generators of $S$ can be chosen with real coefficients.

Finally, by the Oleinik-Petrovsky bound (theorem \ref{CCupperbound}), $\SBetti(T\cap(t+D)) \le \mu(2\mu-1)^{k-1}$. Since this is true generically for $D\in\Grass_{k,n}$ and $t\in\orthog{D}$, we can conclude that $\SBetti(T,k)\le \mu(2\mu-1)^{k-1}$.
\end{proof}

\section{Proof of Optimality of Main Theorem}\label{sect:coun_exam}
In this section we prove the optimality of our main theorem, which we have formalized in our optimality theorem \ref{thm:opti_main}. It is proved by means of counter-examples whose $0^{th}$ Betti number is not controllable by the multiplicity of the germ (hence \textit{a fortiori} neither is $\SBetti$). First, in subsection \ref{exam:coun_exam_pure_dime} we give an example that shows why it is necessary to assume that the complexification of the germ is pure dimensional. This proves the first point in the optimality theorem \ref{thm:opti_main}. Then, in subsection \ref{sect:tang_cone_control} we build counter-example families of every type required by theorem \ref{thm:opti_main} by giving two explicit counter-example families, and introducing two simple transformations on germs to generate the remaining counter-example families.

\subsection{Necessity of pure dimensionality}\label{exam:coun_exam_pure_dime}
\begin{figure}[h]
\centering\def\svgwidth{5cm}%
\executeiffilenewer{example1.svg}{example1.pdf}%
{inkscape -z -D --file=example1.svg %
--export-pdf=example1.pdf --export-latex}%
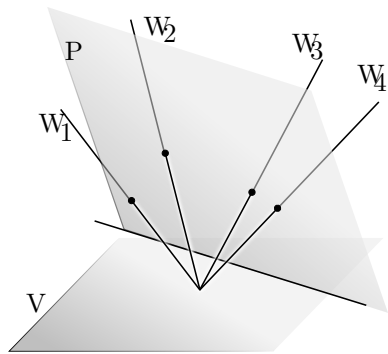%

\caption{The germ consists of a plane $V$ and an arbitrary number of lines $W_i$. Therefore a generic planar section by $P$ has as many connected components as desired because the $W_i$ do not count toward the multiplicity.}
\label{fig:example1}
\end{figure}
The first example shows a phenomenon that is not specific to the real case and that also occurs in the complex case: only the top dimensional components of a germ are taken into account by the multiplicity, but of course, lower dimensional components of the germ can provide connected components when intersecting with an affine space whose dimension is greater than the codimension of the germ. Consequently, these lower dimensional components do not change the multiplicity but can make the number of connected components in a generic section grow arbitrarily large.

Here is one of many ways to construct a family of counter-examples. For any given $l\in\N$, let $V$ be a $d$-dimensional vector space. Let $k\in\N$ such that $n-k<d$, and choose $W_1,\ldots,W_l$, $(n-k)$-dimensional vector spaces such that $\forall i\neq j,\ W_i\cap W_j=\{0\}$ and $W_i\cap V=\{0\}$. Let $X$ be the union of $V$ and all the $W_1,\ldots,W_l$. The situation for $n=3,\ d=2,\ k=2,\ l=4$ is shown in figure \ref{fig:example1}.

The multiplicity of $X$ is clearly $\mu=1$ as an $(n-d)$-dimensional affine space will generically avoid all the $W_1,\ldots,W_l$ (as $(n-d)+(n-k)<(n-d)+d=n$), and the number of points in a generic intersection with $V$ is $1$. On the other hand, a generic $k$-dimensional affine space will intersect $V$ and all the $W_1,\ldots,W_l$, giving rise to $l+1$ connected components in the intersection.

In addition, $X$ and its tangent cone (which are here the same thing) have an isolated singularity at $0$ and their intersection with a generic \mbox{$k$-dimensional} affine space is smooth, therefore the only condition of the main theorem that is not satisfied is that the complexification of the germ should be pure dimensional. This proves the first claim of the optimality theorem \ref{thm:opti_main}.

\subsection{Necessity of controlling the tangent cone}\label{sect:tang_cone_control}
To find counter-example families of all types $(n,k,s)$ as required by the optimality theorem \ref{thm:opti_main}, we use one counter-example family of type $(3,2,1)$, a series of counter-example families of type $(n,n-1,n-1)$, and two transformations that we can apply to the previous germ families to cover all remaining counter-example family types.
\subsubsection{Counter-example family of type (3,2,1)}\label{sssect:coun_fam_321}

\begin{figure}[h]
\centering \def\svgwidth{6cm}%
\executeiffilenewer{example3.svg}{example3.pdf}%
{inkscape -z -D --file=example3.svg %
--export-pdf=example3.pdf --export-latex}%
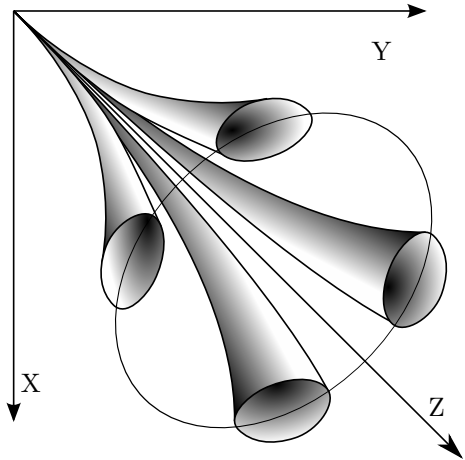%

\caption{The tangent cone is a singular line, but it can arise from any number of higher-dimensional components that degenerate on this line.}
\label{fig:example3}
\end{figure}

The counter-example family of type (3,2,1) is constructed by making ovals degenerate onto a single tangent line as they approach the origin. For any $l\ge 2$, consider 
$$g_l=(x^2+y^2-z^4)^2+\prod_{i=0}^{2l-1} \left(y-\frac{i-l+0.5}{l}z^2\right).$$
Let $X_l$ be the associated germs. The germ $X_3$ is depicted in figure \ref{fig:example3}.

Since we are considering hypersurfaces, the tangent cone is the initial part $(x^2+y^2)^2$  of the generator, which defines a real singular line, and the multiplicity is the degree of that initial part $\mu(X_l)=4$. The complexification of the hypersurface is also automatically pure dimensional.

On planar sections with constant $z$, the $g_l$ can be seen as a circle with radius $z^2$ (left summand) that is perturbed by an oscillating polynomial along the $y$ direction (right summand). Because the right summand is a product of an even number of factors, it is positive for $y<(-1+0.5/l) z^2$ and $y> (1-0.5/l) z^2$. In between its sign alternates on strips of width $z^2/l$. The positive strips separate ovals created in negative strips. The first and last negative strips ($i=0$ to $i=1$, and $i=2l-2$ to $i=2l-1$) contain the extremities $(x=0,y=\pm z^2)$ of the circle defined by the first summand. On all other strips it is elementary to check that the perturbation from the second summand stays smaller than the first summand at $x=0$, which creates two ovals on the circle (since the first summand vanishes on the circle).\\
We thus have $1$ oval for each negative strip at the two extremities of the circle and two ovals for each of the remaining $l-2$ negative strips across the circle, and therefore $2(l-1)$ ovals in total. This shows that $\Betti_0(X_l,2)\ge 2(l-1)$.

\subsubsection{Counter-example families of type (n,n-1,n-1)}\label{sssect:coun_fam_series}
\begin{figure}[h]
	\centering
	\subfloat[frontal view]{\def\svgwidth{.4\columnwidth}%
\executeiffilenewer{example2xy.svg}{example2xy.pdf}%
{inkscape -z -D --file=example2xy.svg %
--export-pdf=example2xy.pdf --export-latex}%
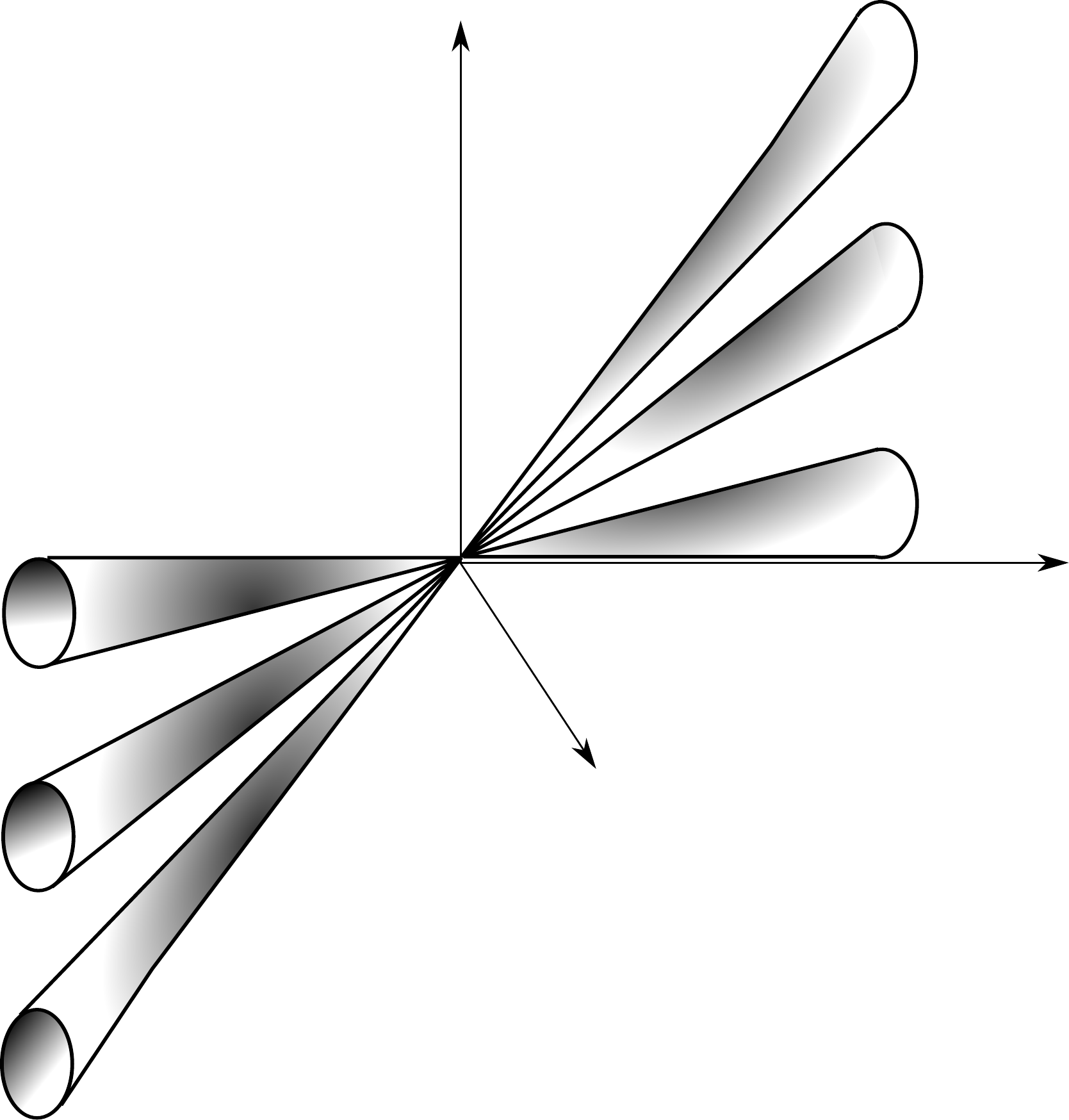%
} 
\hspace{.15\columnwidth}
	\subfloat[side view]{\raisebox{1.4cm}{\def\svgwidth{.4\columnwidth}%
\executeiffilenewer{example2xz.svg}{example2xz.pdf}%
{inkscape -z -D --file=example2xz.svg %
--export-pdf=example2xz.pdf --export-latex}%
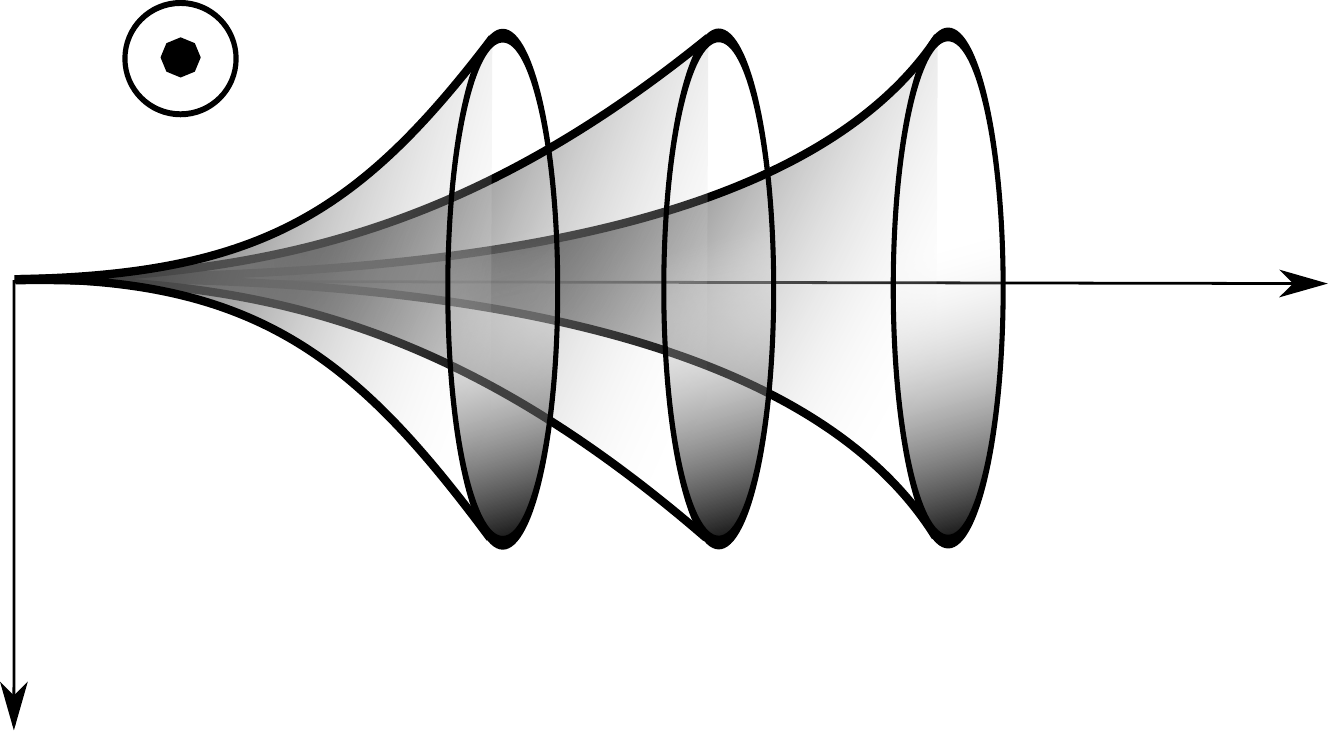%
}}
	\caption{A generic plane cannot avoid the funnel-like branches of the germs, but the algebraic tangent cone is the singular plane $z=0$.}
	\label{fig:example2}
\end{figure}

For any $n\in \N$, $n\ge 3$, we define the families of germ $X_l=V(f_l)\subset \R^n$ by
$$f_l=\prod_{r=0}^{2l-1}(x-ry)+z^2+\sum_{i=1}^{n-3} t_i^4,$$
where the first three components of $\R^n$ are associated to the variables $x,y,z$, the remaining $n-3$ components are associated to the $(t_i)_{i\in\{1,\ldots,n-3\}}$, and the last summand is taken as being zero for $n=3$. Figure \ref{fig:example2} shows what $X_3$ looks like for $n=3$.



We have $T(X_l)=V({\mathrm Init}(f_l))=V(z^2)$, and according to lemma \ref{lem:multtancone} $\mu(X_l)$ is the number of points counted with multiplicities in the intersection of $T(X_l)$ with a generic line. Since $T(X_l)$ is defined by $z^2$ alone, there are $2$ points in a generic section and $\mu(X_l)=2$. Also, as required for this counter-example, the Zariski dimension of $T(X_l)$ is $n-1$ and the complexification of the germ is also automatically pure dimensional since it is a hypersurface.

On any intersecting hyperplane $H$ for $y$ constant, it is easy to verify that the first summand of $f_l$ is negative iff $x \in \cup_{r=0}^{r=l-1}(2ry,2ry+y)$, zero iff $x\in\{0,\ldots,(2l-1)y\}$, and positive elsewhere. Since the sum of the second and third summands is $0$ when $z\!=\!t_1\!=\!\ldots\!=\!t_{n-1}\!=0$, and positive elsewhere, $f_l$ doesn't vanish where the first summand is positive, and is negative at \mbox{$x=2ry+y/2$}, \mbox{$z\!=\!t_1\!=\!\ldots\!=\!t_{n-1}\!=0$} ($r\in\{0,\ldots,l-1\}$). Therefore by continuity, $f_l$ must vanish over each interval of the form $(2ry,2ry+y)$, and the connected components over each interval are distinct from components over other intervals since there are intervening intervals of the form $(2ry+y,2ry+2y)$ where $f_l$ is positive. Since there are $l$ intervals of the form $(2ry,2ry+y)$, this shows that $X_l\cap H$ has at least $l$ connected components (there are in fact exactly $l$ compact connected components, which we do not prove in detail for the sake of brevity).

One can verify elementarily that the number of connected components in the intersection doesn't change for any intersecting hyperplane in a neighborhood of any $y$-constant hyperplane in $\AffGrass_{n-1,n}$. In the spirit of this article we can show this by invoking the fact that the germ can be Whitney stratified, and for a hyperplane $H$ close enough to the origin, $H$ is transverse to all strata. Since transversality is an open condition, by the moving the wall theorem the topology is unchanged for all hyperplanes in a neighborhood of $H$, thereby showing that hyperplane sections of $X_l$ do not generically contains less than $l$ connected components. Hence $\lim_{l\to\infty} \Betti_0(X_l,k)=\infty$.

\subsubsection{Transformations generating the remaining counter-examples}

We now use two simple transformations to generate counter-example families of all types $(n,k,s)$ such that $2\le k\le n-1$, and $n-k\le s \le n-1$ as specified in the optimality theorem \ref{thm:opti_main}.

\begin{defi}[Germ Transformations]\index{Germ transformations}\index{Embedding transformation|\\see{\mbox{Germ transformations}}}\index{Product transformation|\\see{\mbox{Germ transformations}}}\nomenclature{${\mathfrak P}$}{Product transformation}\nomenclature{${\mathfrak E}$}{Embedding transformation}
If $X_l$ is a family of real germs in $\R^n$, the product transformation of $X_l$ is
$${\mathfrak P}(X_l)=X_l\times\R\ \subset\R^{n+1}.$$
The embedding transformation of $X_l$ is
$${\mathfrak E}(X_l)=X_l\times\{0\}\ \subset\R^{n+1}.$$
\end{defi}
\begin{prop}\label{prop:transformations}
If a counter-example family $X_l$ in $\R^n$ is of type $(n,k,s)$ then ${\mathfrak P}(X_l)$ is of type $(n+1,k,s+1)$ and ${\mathfrak E}(X_l)$ is of type $(n+1,k+1,s)$. Also, if $X_l$ is pure dimensional, then the families ${\mathfrak P}(X_l)$ and ${\mathfrak E}(X_l)$ are also pure dimensional.
\end{prop}
\begin{proof}
One can check that the multiplicity of the germs, and the pure dimensionality of the families is preserved by both transformations, as well as the fact that $\Betti_0(X_l,k)= \Betti_0({\mathfrak P}(X_l),k)=\Betti_0({\mathfrak E}(X_l),k+1)$, and that the dimension of the singular locus of the tangent cone is increased by one by ${\mathfrak P}$ and preserved by ${\mathfrak E}$. These proofs are simple, and the details are left to the reader.
\end{proof}

We now use the two transformations to generate all the types of counter-example families we are missing, and prove the second point of the optimality theorem \ref{thm:opti_main}.
\begin{prop}
For any combination of $(n,k,s)\in\N^3$ that satisfies the inequalities $2\le k\le n-1$ and $n-k\le s \le n-1$, there exists a pure dimensional counter-example family $(X_l)_{l\in\N}$ of type $(n,k,s)$.
\end{prop}
\begin{proof}
We work by induction on $n$. We start at $n=3$ since $2\le k \le n-1$ forces $n\ge 3$. For $n=3$ the conditions give two possible types: $(3,2,1)$ which was given by our first counter-example family in \ref{sssect:coun_fam_321} and $(3,2,2)$ which is given by our series of counter-example families in \ref{sssect:coun_fam_series} of type $(n,n-1,n-1)$ for $n=3$.

We now prove the induction step. Let $(n,k,s)\in\N^3$ be a triplet satifying the inequalities in the statement of the proposition.
\begin{itemize}
\item If $k<n-1$, then $(n-1,k,s-1)$ still satisfies the inequalities, and by induction we can find a pure dimensional counter-example family of that type. By applying the product transformation, according to proposition \ref{prop:transformations}, we obtain a counter-example family of type $(n,k,s)$.
\item If $k=n-1$ and $s<n-1$, then $(n-1,k-1,s)$ still satisfies the inequalities, and by induction we can find a pure dimensional counter-example family of that type. By applying the embedding transformation, according to proposition \ref{prop:transformations}, we obtain a counter-example family of type $(n,k,s)$.
\item if $k=n-1$ and $s=n-1$, then the counter-example family series in \ref{sssect:coun_fam_series} provides a suitable family of type $(n,n-1,n-1)$.
\end{itemize}

\end{proof}

%

\section{Bound on Density and Lipschitz-Killing Invariants}\label{sect:appl}
In this section we present a polynomial bound on the local Lipschitz-Killing invariants of a real analytic germ, the first of which is the density  (def. \ref{def:loc_Lip-Kil_var}).
The bound we give on the density and the Lipschitz-Killing invariants of a real analytic germ (prop. \ref{prop:bound_Lip-Kil}) is relevant to the work in \cite{MR1679984,MR1832990,equ_sing_reelle_2} which all involve these quantities. The result is obtained from a localized version of the multidimensional Cauchy-Crofton formula (\cite{equ_sing_reelle_2}, theorem 3.1) and from our main theorem.

 In the first place we define the Lipschitz-Killing invariants by a multidimensional Cauchy-Crofton formula (the density is one of the Lipschitz-Killing invariants). We define the \emph{local} Lipschitz-Killing invariants by a process of limit from the original Lipschitz-Killing invariants. Then we quote theorem 3.1 in \cite{MR1679984} which gives a direct formula for the local Lipschitz-Killing invariant. This latter formula can be seen as a local version of the multidimensional Cauchy-Crofton formula. Finally we apply the main theorem \ref{thm:main} to this expression of the local Lipschitz-Killing invariants to bound them. In order to illustrate the usefulness of these new bounds we conclude by using it to bound the density of an example.

\begin{defi}[$k^\mathrm{th}$ Lipschitz-Killing invariants]\label{defin:Lip-Kil_var}
\index{$k^\mathrm{th}$ Lipschitz-Killing invariant}\nomenclature{${\Lambda}_k$}{$k$-Lipchitz-Killing invariant}
Let $\gamma_{k,n}$ be the unit measure ${\mathcal O}_n(\R)$-invariant on $\Grass_{k,n}$, where ${\mathcal O}_n(\R)$ is the orthogonal group of $\R^n$. Let ${\mathcal H}^k$ be the usual Lebesgue measure on $\R^k$.
For any real sub-analytic set $X$ in $\R^n$, the $k^\mathrm{th}$ Lipschitz-Killing invariant ${\Lambda}^k(X)$ is defined by:
$${\Lambda}_k(X)=\beta(k,n)^{-1}\int_{V\in\Grass_{k,n}}\int_{y\in V} \chi\left(X \cap \pi_V^{-1}(\{y\})\right) d{\mathcal H}^k(y) d\gamma_{k,n}(V),$$
where $\chi$ is the Euler characteristic, $\beta(k,n)$ only depends on $k$ and $n$ and $\pi_V$ is the orthogonal projection to $V$.
\end{defi}
These invariants are to be paralleled with the Vitushkin variations:
\begin{defi}[Vitushkin variations]\label{def:Vitushkin_var}\index{Vitushkin variations}
  For any set $S \subset \RR^n$, let  $V_0 ( S )$ be the number of
  connected components of $S$, and
  \[ V_i ( S ) = c ( i ) \int_{L \in \mathcal{G}_{n - i}} V_0 ( S \cap L )
     \tmop{dL}, \]
  where $\mathcal{G}_k$ is the Grassmannian of affine spaces of dimension $k$
  in $\RR^n,$ $\tmop{dL}$ is the canonical measure on $\mathcal{G}_{n -
  i}$, and 
$c ( i ) =\left(\displaystyle\int_{L \in \mathcal{G}_{n - i}} V_0 ( [ 0, 1 ]^i \cap L )
  \tmop{dL}\right)^{-1}$, (so that $V_i ( [ 0, 1 ]^i ) = 1$ and $c(n)=1$).
\end{defi}
The Vitushkin variations have a more straightforward geometric interpretation since their definition is based on the number of connected components, but the Lipschitz-Killing invariants have more interesting algebraic properties. This is due to the fact that the Lipschitz-Killing invariants enjoy the same additivity property as the Euler characteristic from which they are derived (i.e. $\Lambda_k(X\cup Y)=\Lambda_k(X)+\Lambda_k(Y)-\Lambda_k(X\cap Y)$).

\begin{defi}[Local Lipschitz-Killing invariants and Density]\label{def:loc_Lip-Kil_var}\index{$k^\mathrm{th}$ local Lipschitz-Killing invariant}
\index{$k^\mathrm{th}$ local Lipschitz-Killing invariant}\nomenclature{${\Lambda}^{loc}_k$}{$k$-Local Lipchitz-Killing invariant}
\nomenclature{$\Theta_k(X)$}{$k$-density}
For a real sub-analytic germ $X\subset\R^n$ at the origin, the $k^\mathrm{th}$ local Lipschitz-Killing invariant of $X$ is defined by:
$$\Lambda^{loc}_k(X)=\lim_{r\to 0} \frac{{\Lambda}_k\big(X\cap B^n(0,r)\big)}{{\mathcal H}^k\big(B^k(0,r)\big)},$$
where $B^k(0,r)$ stands for the $k$-dimensional ball centered at the origin of radius $r$. It is proved in \cite{MR1030848} (theorem 2.2) that this limit exists.\\
When $k=\dim(X)$, the fibers $\pi_V^{-1}(\{y\})$ are generically \mbox{$0$-dimensional}, and the Euler characteristic simply counts the number of points in those fibers. In this case the $k^\mathrm{th}$ local Lipschitz-Killing invariant of $X$ is called the $k$-density of $X$ and it is defined as:
$$\Theta_k(X)=\Lambda^{loc}_k(X)=\lim_{r\to 0} \frac{{\Lambda}_k\big(X\cap B^n(0,r)\big)}{{\mathcal H}^k\big(B^k(0,r)\big)}.$$
\end{defi}

At this point, it would already be possible to use the main theorem \ref{thm:main} to derive a bound on the $k$-density and the local Lipschitz-Killing invariants. If we did so we would obtain a bound that is not sharp. Conceptually this would amount to approximating the balls by their bounding cubes. To avoid losing this sharpness we use an alternate characterization of the $k$-density and the local Lipschitz-Killing invariants. It comes from a local version of the multidimensional Cauchy-Crofton formula.

\begin{defi}[Local polar profiles and local polar Euler characteristic]\index{Polar profile}\index{Polar multiplicity}\nomenclature{$K_j^V$}{Polar profile}\nomenclature{$e_j^V$}{Polar multiplicity}
Let $X$ be a real sub-analytic germ at the origin. Let $\mathrm{Crit}(\pi_V|_X)$ be the critical locus of $\pi_V$ on $X_{\mathrm{smooth}}$, where $X_{\mathrm{smooth}}=X\setminus \sing X$. 
Let ${\mathcal O}^k_X$ be the vector spaces $V\in\Grass_{k,n}$ such that $T(X)\cap \orthog{V}=\{0\}$, where $T(X)$ is the tangent cone to $X$.

The local polar profiles of $X$ for $V$ are the connected components $K_j^V$ ($j\in\{0,\ldots,n_V\}$) of the open germ $\pi_V(X)\setminus \pi_V\left(\mathrm{Crit}(\pi_V|_X)\right)$. 
The definition of $\mathrm{Crit}(\pi_V|_X)$ entails that the topology of the fibers $\left( X\cap \pi^{-1}(y)\right)$ ($y\in V$) is constant when $y$ runs over a given $K_j^V$. Therefore, the Euler characteristic of $\left( X\cap \pi^{-1}(y)\right)$ is constant over the $K_j^V$.

We call this constant the local polar Euler characteristic $\chi_j^V$ associated to $K_j^V$. When $\dim(V)=\dim(X)$, the fiber contains finitely many points, and $\chi_j^V$ is thus the number of points in each fiber. In this case, we denote $\chi_j^V$ by $e_j^V$.
\end{defi}
\begin{defi}[Local polar invariants]\index{Local polar invariants}\nomenclature{$\sigma_k$}{Local polar invariants}
Let $X$ be a real sub-analytic germ of $\R^n$. The $k^\mathrm{th}$ local polar invariant associated to $X$ is:
$$\sigma_k(X)= \int_{V\in{\mathcal O}^k_X} \left(\sum_{j=0}^{n_V} \chi_j^V \Theta_k(K_j^V)\right) d\gamma_{k,n}(V).$$
\end{defi}
\begin{theorem}[Local multidimensional Cauchy-Crofton formula]\label{thm:rel_lambda_sigma}
Let $X$ be a real sub-analytic germ of $\R^n$. Theorem 3.1 in \cite{MR1679984} states that there exists an upper-triangular matrix $M\in{\mathcal M}_n(\R)$ such that
$$
\left( \begin{array}{c} \Lambda^{loc}_1\\ \vdots\\ \Lambda^{loc}_n\end{array} \right)
=
\left( \begin{array}{ccccc} 
M_{1,1}& M_{1,2}&\dots&M_{1,n-1}&M_{1,n}\\
0& M_{2,2}&\dots&M_{2,n-1}&M_{2,n}\\
 \vdots&&&&\vdots\\
 0&0&\dots&0& M_{n,n}
 \end{array} \right)
\left( \begin{array}{c} \sigma_1\\ \vdots\\ \sigma_n\end{array} \right)
$$
where $M_{i,i}=1$, and for $i<j\le n$:
$$M_{i,j}= \frac{\alpha_j}{\alpha_{j-i}\ \alpha_i} C_j^i - \frac{\alpha_{j-1}}{\alpha_{j-1-i}\ \alpha_i} C_{j-1}^i,$$
with $\alpha_k$ the $k$-dimensional volume of the unit ball in $\R^k$ and $C_j^i$ the usual binomial coefficients.
\end{theorem}
\begin{rema}
The previous result extends the local Cauchy-Crofton formula for the density in \cite{MR1679984} which asserts that for a $k$-dimensional real analytic germ $X$ at $0$ in $\R^n$:
$$\Theta_k(X)=\int_{V\in{\mathcal O}^k_X} \left(\sum_{j=0}^{n_V} e_j^V \Theta_k(K_j^V)\right) d\gamma_{k,n}(V).$$
This is simply because for $i>k$, the $(n-i)$-dimensional spaces avoid $X$ generically and thus, $\sigma_i=0$. As $M$ is upper-triangular the only non-zero coefficient involved in the expression of $\Theta_k(X)\ (=\Lambda^{loc}_k(X))$ is the diagonal coefficient $M_{k,k}=1$, hence giving the above equality for the density.
\end{rema}
The relation given by theorem \ref{thm:rel_lambda_sigma} between the $\sigma_i$ and the local Lipschitz-Killing invariants enables us to derive the following bounds on the $\Lambda^{loc}_i$:
\begin{prop}[Bound on the $\sigma_i$]\label{lem:bound_sigma}
Let $X\subset\R^n$ be a real analytic germ at the origin. Let $T$ be the algebraic tangent cone to $X$. Then
\begin{enumerate}
\item If $l=\dim\left( X\right)$, we have the inequality 
$$\sigma_l(X)=\Theta_l(X)\le \mu(X).$$
\item For any $l\in\N$, if $\dim\left( \sing{T}\right) < l$ and $\Complex{X}$ is pure dimensional, then 
$$\sigma_l(X)\le \mu(X)(2\mu(X)-1)^{n-l-1}.$$
\end{enumerate}
\end{prop}
\begin{proof}
We start from the expression of $\sigma_l(X)$. Let $k=n-l$. The bound stems from the simple fact that the sum of Betti numbers for a given set is a bound on the Euler characteristic of that set since this latter one is an alternating sum of Betti numbers. Under the hypothesis of the proposition, we can apply the main theorem \ref{thm:main}, and we can conclude that for a generic $V\in\Grass_{l,n}$ and $y\in V$
\begin{align*}
\lim_{\lambda\to 0}\chi\left(X\cap \pi_V^{-1}(\lambda y)\right)&\le \mu(X) &&\text{in case 1},\\
\lim_{\lambda\to 0}\chi\left(X\cap \pi_V^{-1}(\lambda y)\right)&\le \mu(X)(2\mu(X)-1)^{n-l-1} &&\text{in case 2}.
\end{align*}
If $\Theta_l(K_j^V)$ is non-zero, for some $y$ we have a whole segment $(0,y)\subset \pi_V(K_j^V)$. Therefore, for that $y$ we have
$$\lim_{\lambda\to 0}\chi\left(X\cap \pi_V^{-1}(\lambda y)\right)=\chi\left(X\cap\pi_V^{-1}(y)\right)=\chi_j^V.$$
We then replace $\chi_j^V$ in the definition of $\sigma_l(X)$ with $\lim_{\lambda\to 0}\chi\left(X\cap \pi_V^{-1}(\lambda y)\right)$ and use the inequalities above to yield
\begin{align*}
\sigma_l(X)&\le \mu(X) \int_{V\in{\mathcal O}^l_X} \left(\sum_{j=0}^{n_V} \Theta_l(K_j^V)\right) d\gamma_{l,n}(V) &&\text{in case 1},\\
\sigma_l(X)&\le \mu(X)(2\mu(X)-1)^{l-1}\int_{V\in{\mathcal O}^l_X} \left(\sum_{j=0}^{n_V} \Theta_l(K_j^V)\right) d\gamma_{l,n}(V) &&\text{in case 2}.
\end{align*}
Notice that if $\Theta_l(K_j^V)$ was zero, the inequalities on $\chi_j^V$ might not hold (e.g. the counter-example families in section \ref{sect:coun_exam}), but this does not matter as they do not count in the Cauchy-Crofton expression of $\sigma_l(X)$. Finally by definition of the $K_j^V$ we have: 
$$\cup_{j=0}^{n_V} K_j^V = V_0\setminus\pi_V\left(\mathrm{Crit}(\pi_V|_X)\right),$$
where $V_0$ denotes the germ at $0$ associated to $V$. Furthermore the previous union is disjoint and $\pi_V\left(\mathrm{Crit}(\pi_V|_X)\right)$ is of measure $0$. This shows that
\begin{align*}
\sum_{j=0}^{n_V} \Theta_l(K_j^V) &= \Theta_l\Big( V_0\setminus\pi_V\big(\mathrm{Crit}(\pi_V|_X)\big)\Big)\\
 &= \Theta_l(V_0)\quad=\quad 1.
\end{align*}
In the end we obtain the desired result, since ${\mathcal O}^l_X$ is dense in $\Grass_{l,n}$:
\begin{align*}
\sigma_l(X)&\le \mu(X) \int_{V\in{\mathcal O}^l_X} 1\ d\gamma_{l,n}(V)\\
&\le \mu(X) &&\text{in case 1},\\
\sigma_l(X)&\le \mu(X)(2\mu(X)-1)^{l-1}\int_{V\in{\mathcal O}^l_X} 1\ d\gamma_{l,n}(V)\\
&\le \mu(X)(2\mu(X)-1)^{l-1} &&\text{in case 2}.
\end{align*}
\end{proof}
\begin{prop}[Bound on local Lipschitz-Killing invariants]\label{prop:bound_Lip-Kil}
Let $X$ be an analytic germ at $0$ of dimension $d$ such that $\dim\left( \sing{T}\right) < k$. Then we have the following bound on the $k^\mathrm{th}$ local Lipschitz-Killing invariant of $X$:
$$\Lambda^{loc}_k(X)\le M_{k,d}\ \mu(X) + \sum_{l=k}^{d-1} M_{k,l}\ \mu(X)(2\mu(X)-1)^{l-1},$$
where $M$ is defined as in theorem \ref{thm:rel_lambda_sigma}.
\end{prop}
\begin{proof}
By proposition \ref{lem:bound_sigma}, we have a bound in terms of the multiplicity for every $\sigma_l(X)$. We can store those bounds in a vector $b$. According to theorem \ref{thm:rel_lambda_sigma}, the vector $(\Lambda_l(X))$ which is made up of the local Lipschitz-Killing invariants, is the image of the vector $(\sigma_l(X))$ through $M$. The coefficients of $M$ are all non negative, therefore after applying $M$ to the vector $b$, we obtain a vector whose components bound the $\Lambda_l(X)$. By expanding the product $Mb$ we obtain the bounds announced in the proposition. The sum ends at $M_{k,d}$ because the $\sigma_l(X)$ for $l>d$ all vanish (as $(n-l)$-dimensional spaces generically avoid $X$).
\end{proof}
\begin{rema}
For $k=\dim(X)$, the previous proposition gives:
$$\Lambda^{loc}_k(X)\le M_{k,k}\mu(X)=\mu(X).$$
As $\Lambda^{loc}_k(X)=\Theta_k(X)$ by definition, the previous proposition yields the same result as the first item of proposition \ref{lem:bound_sigma} in this case.
\end{rema}

So that the reader can appreciate the sharpness of our bound in view of existing results, we now recall the known Oleinik-Petrovsky/Thom-Milnor bound for the density:
\begin{theorem}[The Oleinik-Petrovsky/Thom-Milnor Bound on Density]\label{thm:thom-milnor_dens_bound}
For an analytic germ $X\subset\R^n$ of dimension $l$ defined by functions $f_1,\ldots,f_k$ of degree $d_1,\ldots,d_k$, 
$$\Theta_l(X)\le (d+1)(2d+1)^{n-l-1}\quad\mathrm{where}\quad d=\sum_{i=1}^k d_i.$$
\end{theorem}
\begin{proof}
This is a straightforward application of theorem 5.5, 
 in \cite{MR2041428} for semi-algebraic sets (using 2 inequalities there to encode 1 equality here).
\end{proof}
Our bound is sharper because it is a localization of the Oleinik-Petrovsky bound: the defining functions may store information on what happens away from the origin, but the multiplicity only accounts for what happens at the origin.

For further illustration, we can work out an example where the variety $X$ is defined by 
$$\big(x(x-z^3)(x-2z^2),  y(y-z^3)(y-2z^2), (x+y)(x+y-z^3)\big).$$
Those generators have degrees $6,6$, and $4$. Therefore by the Oleinik-Petrovsky bound for the density \ref{thm:thom-milnor_dens_bound} we have 
\begin{align*}
\Theta_1(X) &\le (16+1)(2\times 16+1)^{3-1-1}
= 17\times 33 \quad = \quad 561.
\end{align*}
The multiplicity can be shown to be $3$ (see remark \ref{rem:Grobner_multiplicity} for a way to do this using Gr\"obner bases). Therefore, the bound in proposition \ref{lem:bound_sigma} states that
$$\Theta_1(X) \le 3.$$
In fact, the intersection of $X$ with a generic plane has exactly $3$ points in it, therefore the density of $X$ is equal to $3$, and our bound was optimal.

\section*{Acknowledgements}
I thank Prof. G. Comte, my former PhD co-adviser at the University of Nice, for suggesting to me that a bound could be derived on local Betti numbers of real germs with interesting applications to well-known invariants. His help was instrumental not only in guiding me through existing mathematical knowledge during my PhD, but also in providing much needed feedback as I was writing this article.

As most of the work presented in this article was done during my PhD, I must also thank Dr. B. Mourrain, my PhD adviser at the University of Nice, and Prof. E. Bierstone, who supervised me for a year at the University of Toronto, who both contributed to the shaping of the mathematical skills I relied on in the present work.

\bibliographystyle{plain}
\bibliography{germ_bound}
\end{document}